\documentclass{amsart}

\newtheorem{theorem}{Theorem}
\newtheorem{corollary}[theorem]{Corollary}%
\newtheorem{claim}{Claim}%

\theoremstyle{remark}%
\newtheorem{remark}{Remark}%

\begin{document}

\title{Exceptional or half-integral chirally cosmetic surgeries}

\author{Kazuhiro Ichihara}
\address{Department of Mathematics, College of Humanities and Sciences, Nihon University, 3-25-40 Sakurajosui, Setagaya-ku, Tokyo 156-8550, JAPAN}
\email{ichihara.kazuhiro@nihon-u.ac.jp}

\author{Toshio Saito}
\address{Department of Mathematics, Joetsu University of Education, 1 Yamayashiki, Joetsu 943-8512, JAPAN}
\email{toshio@juen.ac.jp}

\dedicatory{Dedicated to Professor Mario Eudave-Mu\~{n}oz on his 60th birthday}

\begin{abstract}
A pair of Dehn surgeries on a knot is called chirally cosmetic if they yield orientation-reversingly homeomorphic 3-manifolds.
In this paper, we consider exceptional or half-integral chirally cosmetic surgeries, and obtain several restrictions.
\end{abstract}



\maketitle

\section{Introduction}\label{sec1}

Generically, the manifolds obtained by Dehn surgeries on a knot are expected to be all distinct. 
However there can exist a pair of such manifolds which are homeomorphic to each other. 
Known such examples are all given by pairs of orientation-reversingly homeomorphic manifolds (mirror images). 
In view of this, a pair of Dehn surgeries on a knot yielding orientation-reversingly homeomorphic manifolds is called \textit{chirally cosmetic}. 
(If the homeomorphism is orientation-preserving, they are called \textit{purely cosmetic}.)

In fact, chirally cosmetic surgeries can occur for some knots in the 3-sphere $S^3$.
The well-known examples are given by the trefoil and the figure-eight knot.

The trefoil is known as the first example of a knot admitting non-trivial chirally cosmetic surgeries.
The prominent ones are 9- and 9/2-surgeries on the right-hand trefoil which create Seifert fibered spaces with opposite orientations.
In Section 2, we show that this example is very special in a sense.
It might be conjectured that only the trefoil can admit such integral and half integral chirally cosmetic surgeries among knots in $S^3$.

The figure-eight knot is actually an amphicheiral knot, i.e., it is ambient isotopic to its mirror image.
Thus it naturally admits chirally cosmetic surgeries.
In fact, for an amphicheiral knot, $r$- and ($-r$)-surgeries give a pair of 3-manifolds which are mirror images of each other for any $r \in \mathbb{Q}$.
On the other hand, the figure-eight knot is well-known in the study of the exceptional surgeries on hyperbolic knots.
For instance, it is recently proved in \cite[Theorem 1.7]{gabai2021hyperbolic} that the figure-eight knot is the unique hyperbolic knot with nine or more exceptional surgeries (actually 10).
It follows that there are 4 pairs of chirally cosmetic exceptional surgeries on the figure-eight knot.
However those are only known examples of chirally cosmetic exceptional surgeries on hyperbolic knots in $S^3$.
In Section 3, we show that this example is also very special in a sense.
Again it might be conjectured that only the figure-eight knot can admit such exceptional chirally cosmetic surgeries among hyperbolic knots in $S^3$.

\begin{remark}
It is conjectured that there are no chirally cosmetic surgeries on a knot $K$ whenever $K$ is a non-trivial knot in $S^3$ other than $(2,p)$-torus knots and amphicheiral knots. (See \cite{IchiharaItoSaito2}.) 
It is also conjectured that there are no purely cosmetic surgeries whenever $K$ is a non-trivial knot in $S^3$. (This is the purely cosmetic surgery conjecture raised in \cite{StipsiczSzabo}.) 
See \cite{Kirby} for more general cases. 
\end{remark}

We here recall some terminology and set up our notations.
Given a 3-manifold $M$ and a knot $K$ in $M$, \textit{Dehn surgery} on a knot $K$ is an operation to make a 3-manifold as $( M - \mathrm{int}N(K) ) \cup_f D^2 \times S^1$, where $N(K)$ denotes a regular neighborhood of $K$. 
Remark that the gluing map $f$ is determined by the isotopy class of a loop on $\partial N(K)$ identified with the meridian of the attached solid torus, called the \textit{surgery slope}.
For knots in $S^3$, or more generally for knots in an integral homology sphere, such surgery slopes can be naturally identified with rational numbers.
See \cite{Rolfsen} for more detail.
Let $K(p/q)$ denote the 3-manifold obtained by Dehn surgery on a knot $K$ in $S^3$ along slope $p/q$.
For oriented 3-manifolds $M$ and $N$, we mean by $M \cong N$ that $M$ and $N$ are orientation-preservingly homeomorphic, and by $-M$ the manifold orientation-reversingly homeomorphic to $M$.
With these notation, Dehn surgeries on a knot $K$ along slopes $p/q$ and $p'/q'$ are chirally cosmetic if and only if $K(p/q) \cong -K(p'/q')$.
Remark that, by homological reason, $p=p'$ must hold in this case.

\section{Half-integral chirally cosmetic surgeries}

One of the well-known non-trivial examples of chirally cosmetic surgeries is given by the Dehn surgeries on the right-hand trefoil in the 3-sphere along slopes $9$ and $9/2$.
See \cite{Mathieu}.
In this section, we show that such chirally cosmetic surgeries along integral and half-integral slopes on knots are quite restricted.

First we have the following.

\begin{theorem}\label{thm1}
Let $K$ be a knot of genus $g$ in the 3-sphere $S^3$. 
For $p>0$, if $K(p) \cong - K(p/2)$, then $p \le 7 g+2$.
\end{theorem}

When $K$ is the right-hand trefoil, the genus $g=1$ and the slopes are $9$ and $9/2$, where $9 \le 7 \cdot 1 + 2$.

On the other hand, the trefoil in $S^3$ is a torus knot and the exterior is homeomorphic to a Seifert fibered space with base orbifold $D^2(2,3)$. 
Concerning this fact, we have the following.

\begin{theorem}\label{thm2}
Let $K$ be a knot in an integral homology sphere with the exterior $E(K) \cong D^2(r,s)$ for some non-zero integers $r, s$.
For $p>0$, if $K(p) \cong - K(p/2)$, then $K=3_1$ in $S^3$ and $p=9$. 
\end{theorem}

\begin{proof}[Proof of Theorem~\ref{thm1}]
Let $K$ be a knot of genus $g$ in the 3-sphere $S^3$.

Suppose that $K(p) \cong - K(p/2)$ for a positive odd integer $p>0$.
Since both $p$ and $p/2$ are positive, by \cite[Proposition 4.2(ii)]{IchiharaItoSaito2}, the surgered manifolds $K(p)$ and $-K(p/2)$ are L-spaces.
That is, $K$ is an L-space knot.

Then, as shown in \cite[Proofs of Theorem 1.2 and 1.4]{WuZhongtao}, $\Delta''_K(1) \ne 0$, where $\Delta_K(t)$ denotes the Alexander polynomial of $K$.
Let $a_2(K)$ denote the second coefficient of the Conway polynomial of $K$.
Since $2 a_2(K) = \Delta''_K(1)$, it implies that $a_2(K) \ne 0$.
Thus we can apply the result \cite[Theorem 2.1]{IchiharaItoSaito1} to have the following.
\[
6 a_2(K) = p( s(1,p) + s(2,p) ).
\]
Here, $s(q, p)$ denotes the Dedekind sum for coprime integers $p, q$ with $p > 0$.
See \cite{IchiharaItoSaito1} for example.
Moreover, by direct calculations, we see that
\[
s(1,p) = \frac{(p-1)(p-2)}{12p} \; , \quad s(2,p) = \frac{(p-1)(p-5)}{24p}.
\]
It follows that
\[
a_2(K) = \frac{(p-1)(p-3)}{48}.
\]

On the other hand, we have the following claim.

\begin{claim}
$a_2(K) \le g^2$
\end{claim}

\begin{proof}
Since $K$ is an L-space knot in $S^3$, as shown in \cite[Proofs of Theorem 1.2 and 1.4]{WuZhongtao}, the following holds.
\[
\Delta''_K(1) = 2 \sum_{j=1}^k (-1)^{k-j} n_j^2
\]
with $0< n_1 < \cdots < n_k$ for some $k$.
Moreover, by \cite[Corollary 1.3]{YiNi}, $K$ is fibered and so, we have $n_k=g$.

If $k=1$, then
\[
a_2(K) = n_1^2 = g^2.
\]
Otherwise, from $ (-1) n_l^2 + n_{l-1}^2 < 0 $ for $l = k-1, k-3, \cdots$, we obtain that
\[
a_2 (K) = \frac{1}{2} \Delta''_K(1) = \sum_{j=1}^k (-1)^{k-j} n_j^2 < n_k^2 = g^2
\]
\end{proof}

Now, assume for a contradiction that $p > 7 g(K)+2$.
Then we would have
\[
a_2(K) = \frac{(p-1)(p-3)}{48} > \frac{49g^2-1}{48}.
\]
This contradicts $a_2(K) \le g^2$ shown in the claim above and $g \ge 1$.
\end{proof}

\begin{proof}[Proof of Theorem~\ref{thm2}]
Let $X$ be the exterior of $K$ which is a Seifert fibered space with orientable orbifold $B$ and
projection map $\pi:X \to B$. Let $\nu_1, \ldots , \nu_n$ be the singular points of $B$ and
$\delta_1, \ldots , \delta_n$ their pairwise disjoint regular neighborhoods respectively.
We set $X_0=\mathrm{cl} \left( X \setminus \bigcup_{i=1}^{n} \pi^{-1}(\delta_i) \right)$
and $B_0=\pi (X_0)=\mathrm{cl} \left( B\setminus \bigcup_{i=1}^{n} \delta_i \right)$.
Let $T_0, \ldots , T_n$ be boundary tori of $X_0$ defined by $T_0=\partial X$ and $T_i=\partial \pi^{-1}(\delta_i)$
$(i\ne 0)$.

Under the assumption that $K$ admits chirally cosmetic surgeries, 
it follows from \cite[Theorem 1]{Rong} that under some choice of the section $\sigma:B_0 \to X_0$,
the unnormalized Seifert invariants of singular fibers of $X$ are

\[
\left\{
\frac{\beta_1}{\alpha_1},\frac{\beta_2}{\alpha_2},-\frac{\beta_2}{\alpha_2}, \ldots ,
\frac{\beta_j}{\alpha_j},-\frac{\beta_j}{\alpha_j}, \frac{1}{2}, \ldots , \frac{1}{2}
\right\}, \
\mathrm{where}\ \frac{\beta_i}{\alpha_i}\ne 0, \frac{1}{2}\ (\bmod \ 1).
\]

Suppose that $j>1$ and let $h$ be a regular fiber of $X$ which generates a cyclic normal subgroup of $\pi_1(X)$.
For each boundary torus $T_i$, set $c_i=\sigma(B_0)\cap T_i$. Then we have 
\begin{eqnarray*}
H_1 (X) \cong \Bigg\langle c_0, c_1, \ldots, c_n, h
\left|
\begin{array}{l}
\alpha_1 c_1 - \beta_1 h =0, \ \alpha_{k+1} c_{2k} - \beta_{k+1} h =0, \\
\alpha_{k+1} c_{2k+1} + \beta_{k+1} h =0, \ 2c_l - h = 0, \\
c_0 + c_1 + c_2 + \cdots + c_n = 0
\end{array}
\right\rangle ,
\end{eqnarray*}

\noindent
where $k=1, \ldots , j-1$ and $l=2j, \ldots , n$.
In particular, we obtain $\alpha_{k+1} (c_{2k} + c_{2k-1}) = 0$ from
two relations $\alpha_{k+1} c_{2k} - \beta_{k+1} h =0$ and $\alpha_{k+1} c_{2k+1} + \beta_{k+1} h =0$.
Under the assumption that $K$ is a knot in an integral homology sphere, we see that $H_1(X)\cong \mathbb{Z}$,
and therefore, $H_1(X)$ has no torsion elements.
Hence we see that $c_{2k} = - c_{2k+1}$ $(k=1,2, \ldots , j-1)$.
Similarly, we see that $c_{2j} = c_{2j+1} = \cdots = c_n$ and $h=2c_n$. 
This implies:
\begin{eqnarray}\label{HomPre}
H_1 (X) \cong \Bigg\langle c_1, c_3, \cdots,c_{2j-1}, c_n
\left|
\begin{array}{l}
\alpha_1 c_1 - 2 \beta_1 c_n =0, \\
\alpha_{k+1} c_{2k+1} + 2 \beta_{k+1} c_n =0
\end{array}
\right\rangle .
\end{eqnarray}

\noindent
As in the assumption of Theorem~\ref{thm2}, we hereafter assume that the exterior $X$ is a Seifert fibered space with two singular fibers. 
Then the Seifert invariants of $X=E(K)$ are
\[
\left\{
\frac{\beta}{\alpha},\frac{1}{2}
\right\}, \
\mathrm{where}\ \frac{\beta}{\alpha}\ne 0, \frac{1}{2}\ (\bmod \ 1).
\]

\noindent
The arguments of \cite[Appendix]{IchiharaItoSaito1} show that if $K(p) \cong - K(p/2)$, then

\begin{eqnarray*}
p=\frac{\alpha^2(2m+1)}{\frac{1}{2}(\alpha(2m+1)-1)}, \ \ \ \
p/2=\frac{\alpha^2(2m+1)}{\frac{1}{2}(\alpha(2m+1)+1)}
\end{eqnarray*}

\noindent
for some integer $m$.
This implies that $\alpha ( 2m+1) = 3$.
Since $\alpha \ne \pm 1$, we see that $(\alpha , m) = (3,0)$ or $(-3,-1)$.
We therefore obtain $p = 9$ under the assumption that $p>0$.

We also see that $X$ is the exterior of a trefoil as follows. 
There is a horizontal normal surface $F$ in $X$ that is a degree 6 cover of the underlying orbifold $B$. 
See \cite[Proposition 3.7]{jackson2023recognition} for example. 
Then it follows from Riemann-Hurwitz formula that the Euler characteristic of $F$ is $-1$, i.e., $F$ is a once-punctured torus.
This means that $X$ is a once-punctured torus bundle over $S^1$.
Using the homology of $X$ as with Burde and Zieschang \cite{zbMATH03273014}, we see that $K$ is a trefoil.
\end{proof}

\begin{remark}
Let $M$ be a closed oriented 3-manifold and 
$K \subset M$ a knot such that its exterior is a Seifert fibered space 
with base orbifold a disk and more than two singular fibers. 
If $M$ is a rational homology sphere, then we can find such a knot $K$ 
which admits $K(p) \cong -K(p/2)$ for an integer $p$. 

To this end, let $X$ be a Seifert fibered space with base orbifold a disk 
such that the Seifert invariants are written as 
$\left\{ -1/3,1/5,-1/5,1/2 \right\}$
under some choice of the section. 
Then it follows from \eqref{HomPre} in the proof of Theorem~\ref{thm2} that 
\begin{eqnarray*}
H_1 (X) \cong 
\langle \, c_1, c_3, c_4 \ | \ 
3c_1=5c_3=-2c_4 \,
\rangle .
\end{eqnarray*}

Adding $z = c_1+c_3+c_4$ as a generator of $H_1(X)$, we see that 
$H_1(X) \cong \langle \, z \ | \ \textendash \, \rangle \cong \mathbb{Z}$ 
and $c_0 = 5z$ and $h = -30z$ in $H_1(X)$. 
We now let $\mu$ and $\lambda$ be simple closed curves in $\partial X$ 
such that $\mu = c_0$ and $\lambda = 6c_0 + h$ in $H_1(\partial X)$. 
Then the Dehn filling along $\mu$ yields a closed $3$-manifold $M$ with 
$H_1(M) \cong \mathbb{Z}/5 \mathbb{Z}$. 
Hence the core loop, say $K$,  of the filling solid torus is a knot in 
the rational homology sphere $M$, and $(\mu, \lambda)$ is a 
meridian-longitude pair of $K$ in $\partial E(K) = \partial X$. 
It follows from \cite[Theorem 1]{Rong} that after changing the basis 
in our setting, the following two slopes 
\[
\gamma_1=\dfrac{18m+9}{3m+1}, \ \ \gamma_2=\dfrac{18m+9}{3m+2}
\]

\noindent
give chirally cosmetic surgeries. In particular, the case that $m=0$ shows that 
$K(9) \cong -K(9/2)$. 

At this moment, the authors do not know whether there exists such a knot, except for a trefoil, 
when $M$ is an integral homology sphere. 
\end{remark}

\section{Chirally cosmetic exceptional surgery slopes}

In this section, we consider chirally cosmetic \textit{exceptional} surgeries on hyperbolic knots in the 3-sphere $S^3$, i.e., a pair of Dehn surgeries on a hyperbolic knot yielding orientation-reversingly homeomorphic non-hyperbolic 3-manifolds. 
We show that the surgery slopes for such surgeries are quite restricted, in particular, possible integral ones are only $\pm1, \pm 2, \pm 3$, or $\pm 4$. 
All these slopes actually appear for the figure-eight knot. 
As far as the authors know, only examples of chirally cosmetic exceptional surgeries are on the figure-eight knot.

\begin{theorem}\label{thm3}
Let $K$ be a hyperbolic knot in $S^3$.
Suppose that $K$ admits chirally cosmetic exceptional surgeries, i.e., $K(p/q) \cong - K(p/q')$ for $p,q>0$, $q,q'$ are coprime to $p$ with $q \ge q'$ and $K(p/q)$ is non-hyperbolic.
Then $q=-q'$ and $p/q= 1,2,3,4$, $1/2,1/3$ or $1/4$.
\end{theorem}

As stated above, all the integral ones in the theorem can be realized by the figure-eight knot.

Furthermore, also shown is the figure-eight knot is the only knot admitting such surgeries among all the hyperbolic alternating knots and Montesinos knots.

\begin{theorem}\label{thm4}
Let $K$ be a hyperbolic knot in $S^3$.
Suppose that $K$ is an alternating knot or a Montesinos knot and admits chirally cosmetic exceptional surgeries.
Then $K$ is the figure-eight knot and the surgeries slopes are $\pm 1, \pm 2, \pm 3, \pm 4$.
\end{theorem}

\begin{remark}
It is shown in \cite{Ravelomanana} 
that purely cosmetic and exceptional surgeries on hyperbolic knots in $S^3$ are quite restricted.
In particular, algebraic knots, alternating hyperbolic knots, and arborescent knots are shown to have no purely cosmetic exceptional surgery.
\end{remark}

It is a natural question if the other knot can admit such surgeries.
This is related to a question if any hyperbolic amphicheiral knot can have non-longitudinal exceptional surgeries, which would be still open.
The next is an immediate corollary of the theorem above.

\begin{corollary}\label{cor5}
Among hyperbolic alternating knots and Montesinos knots, the figure-eight knot is the only knot which is amphicheiral and admits non-longitudinal exceptional surgeries.
\end{corollary}

Among the exceptional slopes for the figure-eight knot, $\pm 4$ are toroidal and the others are all Seifert surgeries.
These Seifert fibered manifolds obtained are all small (atoroidal) Seifert fibered with infinite fundamental groups.

About the type of the manifolds obtained by chirally cosmetic exceptional surgeries, we have the following.

\begin{corollary}\label{cor6}
Let $K$ be a hyperbolic knot in $S^3$.
Suppose that $K$ admits chirally cosmetic exceptional surgeries.
Then the obtained manifolds are irreducible and have infinite fundamental groups.
Furthermore if the obtained manifolds are toroidal, then the surgery slopes are $\{ \pm 4 \}$, which actually occur for the figure-eight knot, or $\{ \pm 1 \}$.
\end{corollary}

\begin{proof}[Proof of Theorem~\ref{thm3}]
Let $K$ be a hyperbolic knot in $S^3$.
Suppose that $K(p/q) \cong - K(p/q')$ for $p,q>0$, $q,q'$ are coprime to $p$ with $q \ge q'$ and $K(p/q)$ is non-hyerpbolic. 

For exceptional slopes $p/q$ and $p/q'$, by the result of \cite{LackenbyMeyerhoff}, the distance $\Delta (p/q, p/q')$ between them is at most eight, i.e.,
\[
\Delta (p/q, p/q') = p | q' - q | \le 8 .
\]
Now we consider two cases $q' >0$ and $q' < 0$.

Suppose that $q' <0$.
Then, for coprime integers $q,q'$, we have $| q' - q | \ge 3$.
It implies that $p \le 2$, which contradicts to \cite[Corollary 1.6]{IchiharaItoSaito2} which says that $p>2$.

Suppose that $q' >0$.
Then, by \cite[Proposition4.1(ii)]{IchiharaItoSaito2}, 
the obtained manifold is an L-space, i.e., $K$ is an L-space knot.
It implies that $a_2(K) \ne 0$ by \cite{WuZhongtao}. 
From $p \le 8$ together with $a_2(K) \ne 0$, we can apply \cite[Corollary 2.2]{IchiharaItoSaito1} 
to obtain the following:
\begin{itemize}
\item $p=7$, $q=7s+1$, $q'=-7s-2$  $(s\in \mathbb{Z})$, or 

\item $p=7$, $q=7s+2$, $q'=-7s-1$  $(s\in \mathbb{Z})$.
\end{itemize}
In both cases, we have
\[
\Delta (p/q, p/q') = p | q' - q | = 7 ( 14 s + 3 ) \qquad (s \in \mathbb{Z})
\]
which contradicts $\Delta (p/q, p/q') \le 8$.

As a consequence, $q' = -q$ must hold.
Then, by $ p | q' - q | \le 8 $ and $p,q$ are coprime, we obtain $p/q= 1,2,3,4$, $1/2,1/3$ or $1/4$.
\end{proof}

\begin{proof}[Proof of Theorem~\ref{thm4}]
Suppose first that a hyperbolic alternating knot $K$ admits chirally cosmetic exceptional surgeries, that is, $K(p/q) \cong - K(p/q')$ for $p>0$, $q,q'$ are coprime to $p$ with $q \ge q'$ and $K(p/q)$ is non-hyperbolic.
Then, by \cite[Corollary 1.2.]{IchiharaMasai}, 
one of the following holds.

\begin{itemize}
\item
$K$ is equivalent to the figure-eight knot, and the exceptional slopes are;
\[ -4, -3, -2,-1, 0, 1, 2, 3, 4. \]
\item
$K$ is equivalent to a two bridge knot $K_{[2n, 2]}$ (resp. $K_{[2n, -2]}$) with $ | n | > 2$, and the exceptional slopes are;
\[ -4, -3, -2,-1, 0. \qquad (\text{resp. } 0, 1, 2, 3, 4 .) \]
\item
$K$ is equivalent to a two bridge knot $K_{[b_1,b_2]}$ with $ | b_1 | , | b_2 | > 2$ and both $b_1$ and $b_2$ are even (resp. $b_1$ is odd and $b_2$ is even), and the exceptional slopes are $ 0$ (resp. $ 2 {b_2}$) only.
\item
$K$ is equivalent to a pretzel knot $P(q_1,q_2,q_3)$ with $q_j \ne 0, \pm 1$ for $j = 1,2,3$ and $q_1,q_2,q_3$ are all odd (resp. $q_1$ is even and $q_2,q_3$ are odd), and the exceptional slopes are $ 0$ (resp. $2(q_2 + q_3)$) only.
\end{itemize}

This shows that the only possibility for $K$ is the figure-eight knot by Theorem~\ref{thm3}.
In fact, the figure-eight knot is amphicheiral, and so, $\pm 1, \pm 2, \pm 3, \pm 4$-surgeries are actually chirally cosmetic and exceptional.

Suppose next that a hyperbolic Montesinos knot $K$ admits chirally cosmetic exceptional surgeries, that is, $K(p/q) \cong - K(p/q')$ for $p>0$, $q,q'$ are coprime to $p$ with $q \ge q'$ and $K(p/q)$ is non-hyperbolic.

Then the length of $K$ must be at most three by \cite[Theorem 3.6]{WuYingQing1996} 
and if $K$ is alternating, then the above arguments assure that $K$ is the figure-eight knot.
Thus, in the following, we assume that $K$ is non-alternating, which implies the length of $K$ is just three, i.e., $K = M ( p_1/q_1, p_2/q_2, p_3/q_3 )$ with $q_j \ge 2$ for $j=1,2,3$.
Also, since Montesinos knots have no reducible surgeries by \cite[Corollary 2.6]{WuYingQing1996}, 
each exceptional surgery slope for $K$ is a toroidal slope or a Seifert slope.
Remark that there are no toroidal Seifert slopes for Montesinos knots other than the trefoil \cite[Theorem 1.1]{IchiharaJong}. 

In the case that $K$ admits Seifert surgeries, by \cite[Appendix B]{IchiharaMasai}, 
one of the following holds.

\begin{table}[htb]
\caption{Montesinos knots with Seifert surgeries}\label{tab:my_label}%
\begin{tabular}{c|c}
Montesinos knot& exceptional slopes \\
\hline
$P(-2, 3, 2n+1)$ with $n \ne -1,0,1,2$ & $4n+6, 4n+7, 4n+8$ \\
$P(-2, 3, 7)$ & $16, 17, 18, 37/2, 19, 20$ \\
$P(-3, 3, 3)$ & $0, 1, 2$ \\
$P(-3, 3, 4)$ & $0, 1$ \\
$P(-3, 3, 5)$ & $0, 1$ \\
$P(-3, 3, 6)$ & $0, 1$ \\
$M(-1/2, 1/3, 2/5)$ & $3, 4, 5, 6$ \\
$M(-1/2, 1/3, 2/7)$ & $-2, -1, 0, 1$ \\
$M(-1/2,1/3,2/9)$ & $2, 3, 4, 5$ \\
$M(-1/2, 1/3, 2/11)$ & $-3, -2, -1, 0$ \\
$M(-1/2, 1/5, 2/5)$ & $7, 8, 9$ \\
$M(-1/2, 1/7, 2/5)$ & $11, 12$ \\
$M(-2/3, 1/3, 2/5)$ & $-6, -5, -4$\\
\end{tabular}
\end{table}

Due to Theorem~\ref{thm3}., this implies that the only possibility for $K$ is the Montesinos knot $M(-1/2, 1/3, 2/7)$ with slopes $\pm 1$.
For this knot, it is shown that $\pm 1$-surgeries are not chirally cosmetic in several ways.
For example, this knot is known as $10_{132}$ in the knot table.
Such a knot can be checked by using a computer whether it has cosmetic surgeries based on the method developed in \cite{FuterPurcellSchleimer}. 
Also see \cite[Remark 1.14]{IchiharaItoSaito2}. 
It is also checked by using finite type invariants.
Actually, by direct computations, we see that the value of $v_3$ for the knot is $-5 \ne 0$, and so, $\pm 1$-surgeries are not cosmetic by \cite[Corollary 1.3(ii)]{Ito}. 
More concretely, in \cite[Theorem 1.4]{Meier}, 
the following is obtained.
\begin{align*}
M(-1/2, 1/3, 2/7) (-1) & \cong S^2(1/3,1/4,-4/7) \\
M(-1/2, 1/3, 2/7) (1) & \cong S^2(1/2,1/3,-16/19)
\end{align*}
Thus they are not homeomorphic.
It concludes that $K$ has no chirally cosmetic exceptional surgeries.

In the case that $K$ has only toroidal slopes, we also see that $K$ has no chirally cosmetic exceptional surgeries as follows.
It suffices to consider that $K$ has at least two toroidal slopes (and no Seifert slopes).
Then, by \cite[(3) in page 308]{WuYingQing2011} 
$K$ admits exactly two toroidal surgeries, and $K$ is equivalent to one of five knots.
Three of the five admit Seifert slopes, and so, already appear in the above table.
The remaining are the following.
\begin{itemize}

\item
$K$ is equivalent to $P(-3,3,7)$ and the exceptional (i.e., toroidal) slopes are;
\[
0, 1.
\]

\item
$K$ is equivalent to $M(-2/3, 1/3, 1/4)$, and the exceptional (i.e., toroidal) slopes are;
\[
12, 13.
\]

\end{itemize}
These pairs of slopes cannot be chirally cosmetic by Theorem~\ref{thm3}, and so, both knots have no chirally cosmetic exceptional surgeries.
\end{proof}

\begin{proof}[Proof of Corollary~\ref{cor6}]
Let $K$ be a hyperbolic knot in $S^3$ which admits chirally cosmetic exceptional surgeries, i.e., $K(p/q) \cong - K(p/q')$ for $p,q>0$, $q,q'$ are coprime to $p$ with $q \ge q'$ and $K(p/q)$ is non-hyperbolic.

It is known by \cite{GordonLuecke1987} and \cite{CGLS} that if the manifolds obtained by two Dehn surgeries on a knot are reducible or have cyclic fundamental groups, then the distance between the surgery slopes is at most one.
Thus, by Theorem~\ref{thm3}, the manifolds $K(p/q)$ and $K(p/q')$ are irreducible and do not have cyclic fundamental groups.

In the case that the obtained manifolds have the finite but noncyclic fundamental groups, then such a pair of surgery slopes are classified in \cite[Theorem 1.4(3)]{NiZhang}.
The list does not contain any of the slope pairs given in Theorem~\ref{thm3}.
This implies that $K(p/q)$ and $K(p/q')$ have the infinite fundamental groups.

Suppose next that the obtained manifold is toroidal.
Note that, in the list of the slopes given in Theorem~\ref{thm3}, only $\{ \pm 1 \}$ are of distance 2, and the others have distance at least 4.
On the other hand, in \cite{GordonWu}, toroidal surgeries on hyperbolic knots in $S^3$ with distance at least 4 are completely classified.
Among them, we see that only $\{ \pm 4 \}$-surgeries on the figure-eight knot can be chirally cosmetic due to Theorem~\ref{thm3}, and actually they are.
\end{proof}

\section*{Acknowledgments}
The authors thank Tetsuya Ito for useful comments on Theorem~\ref{thm1} and Masakazu Teragaito for comments on Corollary~\ref{cor5}.
This work was supported in part by JSPS KAKENHI Grant Numbers JP22K03301 and JP21K03244.

\bibliographystyle{plain}
\bibliography{IchiharaSaito}

\end{document}